\documentclass[none]{apjrnl}
\usepackage[finnish,english]{babel}
\usepackage[dvips]{graphicx}
\usepackage{amsfonts}
\usepackage{amssymb}
\usepackage{amsmath}
\usepackage{latexsym}
\usepackage[latin1]{inputenc}

{\theoremstyle{plain}%
  \newtheorem{theorem}{Theorem}
  \newtheorem{corollary}{Corollary}
  \newtheorem{proposition}{Proposition}
  \newtheorem{lemma}{Lemma}%

{\theoremstyle{remark}

}
{\theoremstyle{definition}
\newtheorem{definition}{Definition}
\newtheorem{example}{Example}
}

\begin{document}

\newcommand{\EEr}{{\mathcal E}_{r+1}}
\newcommand{\EErr}{{\mathcal E}_{r}}
\newcommand{\E}{{\mathcal E}}
\newcommand{\F}{{\mathcal F}}

\newcommand{\dunion}{\amalg}

\title{The Algebra of Graph Invariants - Upper and Lower Bounds for Minimal Generators}
\author{Tomi Mikkonen\\ Tampere University of Technology\\ Department of Signal
  Processing \\
  P.O.BOX 553\\
    FIN-33101 TAMPERE \\
    FINLAND\\
\small{tomi.mikkonen@tut.fi}\normalsize
\and Xavier Buchwalder\\ Universit\'e de Lyon\\ Institut Camille Jordan UMR
5208\\
Universit\'e Lyon 1\\
43 boulevard du 11 Novembre 1918\\
69622 Villeurbanne cedex\\
FRANCE\\
\small{buchwalder@math.univ-lyon1.fr}\\
}

\maketitle

\begin{abstract}
In this paper we study the algebra of graph invariants, focusing mainly on the
invariants of simple graphs. 

All other invariants, such as sorted eigenvalues, degree sequences and canonical permutations, belong to this algebra. In fact, every graph invariant is a linear combination of the basic graph invariants which we study in this paper.

To prove that two graphs are isomorphic, a number of basic invariants are required,
which are called separator invariants. The minimal set of separator invariants is
also the minimal basic generator set for the algebra of graph invariants. 

We find lower and upper bounds for the minimal number of generator/separator
invariants needed for proving graph isomorphism. 

Finally we find a sufficient condition for Ulam's conjecture to be true based
on Redfield's enumeration formula.
\end{abstract}

\section{Introduction}
Let $G=(V,E)$ and $H=(V,F)$ be simple graphs (i.e. unoriented, no loops, no multiple edges) with $n$ vertices, where $V$ is the common set of vertices and $E,F$ are the sets of edges. We say that $G$ is a subgraph of $H$ if $E\subseteq F$. Two graphs $G$ and $H$ are isomorphic, denoted by $G \cong H$, if there exists a permutation $\pi$ of the set of vertices such that $\pi E=F$. 

In this paper we study \emph{basic graph invariants} which count the number of
subgraphs isomorphic to $G$ in $H$. We denote by $I(G)(H)$ the number of
subgraphs isomorphic to $G$ in the graph $H$. Graphs are denoted by monomials
$\prod_{(i,j)\in E} a_{ij}$. Thus for instance
$I(a_{12})(a_{23}a_{24}a_{34})=3$ and
$I(a_{12}a_{23})(a_{23}a_{24}a_{34})=3$. The definition of $I(G)(H)$ depends
only on the isomorphism classes of $G$ and $H$.

Let $A$ be the adjacency matrix of the graph $H$, i.e. $a_{ij}=1$ if there is an edge between vertices $i$ and $j$ in $H$ and $a_{ij}=0$ otherwise. Because $H$ is an unoriented graph we have $a_{ij}=a_{ji}$. Then $I(G)(H)$ is a function in the variables $a_{ij}$ and can be written explicitly as 
\begin{equation}\label{eq:nollaa}
I(G)(H)=\frac{1}{|\mathrm{Stab}(a_{i_1j_1}a_{i_1j_2}\cdots a_{i_dj_d})|}\sum_{\rho \in S_n} a_{\rho(i_1)\rho(j_1)}a_{\rho(i_2)\rho(j_2)}\cdots a_{\rho(i_d)\rho(j_d)}.
\end{equation}
Here the $(i_kj_k)$-pairs correspond to the edges in some labeling of the graph $G$. The stabilizer is 
\begin{equation}
\mathrm{Stab}(a_{i_1j_1}\cdots a_{i_dj_d})=\{\rho \in S_n |  a_{\rho(i_1)\rho(j_1)}\cdots a_{\rho(i_d)\rho(j_d)}=a_{i_1j_1}\cdots a_{i_dj_d}\}
\end{equation}
with respect to the symmetric group $S_n$, where two monomials are considered
the same if they have the same variables. The use of the stabilizer in
equation (\ref{eq:nollaa}) guarantees that the coefficient of each monomial in
the sum is one. Note that every monomial is either $1$ or $0$ depending on
whether the monomial is contained in $H$. The total degree $d$ of the
polynomial, denoted also as $|G|$, corresponds to the number of edges in
$G$. Examples of these so-called \emph{orbit sums} are shown below.

The sum in (\ref{eq:nollaa}) clearly permutes the monomial $a_{i_1j_1}\cdots a_{i_dj_d}$ in every possible location in the vertex set $V$. Thus $I(G)(H)$ depends only on the isomorphism classes of the $G$ and $H$, i.e. it is invariant with respect to the labeling of the vertices. 

Consequently, in the definition of $I(G)(H)$ we can represent the graph $G$ using vertex-edge sets, the adjacency matrix, a monomial in the variables $a_{ij}$ or a graphic representation. For example the vertex-edge set $(\{1,2,3,4\},\{\{1,2\},\{1,3\},\{2,3\}\})$, the adjacency matrix
\begin{displaymath}
\left[ 
\begin{array}{llll}
0 & 1 & 1 & 0 \\
1 & 0 & 1 & 0 \\
1 & 1 & 0 & 0 \\
0 & 0 & 0 & 0
\end{array} \right],
\end{displaymath}
the monomial $a_{12}a_{13}a_{23}$ and Figure \ref{fig:0} represent the same
graph. 
\begin{figure}[!htk]
\begin{center}
\includegraphics[width=1cm,height=1.5cm]{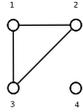}
\end{center}
\vspace{-0.5cm}
\caption{Graphic representation.}\label{fig:0}
\end{figure}

We can express the sum (\ref{eq:nollaa}) without division by using the quotient of
groups as follows
\begin{equation}\label{eq:eqeka}
I(G)(H)=\sum_{\rho \in S_n/\mathrm{Stab}(a_{i_1j_1}a_{i_1j_2}\cdots a_{i_dj_d})} a_{\rho(i_1)\rho(j_1)}a_{\rho(i_2)\rho(j_2)}\cdots a_{\rho(i_d)\rho(j_d)}.
\end{equation}
This representation remains valid in fields of finite characteristic and we
will use this as the definition of $I(G)(H)$.

We present examples mostly in the algebra $\left( \mathbb{C}[a_{ij}]/\langle
    a_{ij}^2-a_{ij},a_{ij}-a_{ji} \rangle \right)^{S_n}$ but all results generalize directly to general permutation
    groups $G$. Also generalization to $(\otimes_d \mathbb{Z}_+^n)^{G}$ is
    quite obvious and only partially presented here as the purpose of this
    paper is to provide an invariant theoretical view to classical graph theory.

\begin{example}
Choose $d=1$, $(i_1j_1)=(12)$, $n=4$, corresponding to $G=a_{12}$. The stabilizer for $a_{12}$ is $\mathrm{Stab}(a_{12})=\{ (1234),(2134),(1243),(2143) \}$ and 
\begin{equation}
S_n/\mathrm{Stab}(a_{12})=\{(1234),(1324),(1423),(2314),(2413),(3412)\}.
\end{equation}
The orbit sum (\ref{eq:eqeka}) becomes in this case
$I(a_{12})(H)=a_{12}+a_{13}+a_{14}+a_{23}+a_{24}+a_{34}$ and it calculates the number of edges in a graph with $4$ vertices. It is clearly invariant with respect to all permutations of vertices.
\end{example}
\begin{example}
Choose $d=2$, $(i_1j_1)=(12)$, $(i_2j_2)=(13)$, $n=4$.\\
The invariant $I(a_{12}a_{13})(H)=a_{12}a_{13}+a_{12}a_{14}+a_{12}a_{23}$
$+a_{12}a_{24}+a_{13}a_{14}+a_{13}a_{23}+a_{13}a_{34}+a_{14}a_{24}$
$+a_{14}a_{34}+a_{23}a_{24}+a_{23}a_{34}+a_{24}a_{34}$
calculates the number of subgraphs isomorphic to $a_{12}a_{13}$ in the graph $H$. Any permutation of vertices affects only the order of summation. 
\end{example}
 We call the polynomials $I(G)(H)$ \textit{basic graph invariants of type $G$}. The basic graph invariants $I(G)(H)$ are polynomials in the variables $a_{ij}$ and depend on $H$ only through the values of these variables. Thus we may consider the basic graph invariants as symbolic polynomials in $a_{ij}$ and we often drop the second graph ($H$ in $I(G)(H)$) from the notation. We use the notation $I(A)$ for this symbolic polynomial, where $A$ is some monomial in the orbit sum.

In \cite{Fleischmann} Fleischmann describes a general formula for the product
of two orbit sums in a graded algebra. In this paper we will modify this
product formula so that it calculates the product of two basic graph
invariants, i.e. 
\begin{equation}
I(A)\cdot I(B)=\sum_{k} c_{A,B}^k I(G_k),
\end{equation}
as a linear combination of basic graph invariants $I(G_k)$. The result is
closely similar to Kocay's lemma \cite{Kocay2} which also gives the coefficients $c_{A,B}^k$. In sections \ref{sec:se2} and \ref{sec:se4}, we introduce two distinct ways of calculating these coefficients. However the second looks more efficient, we still believe that the first is of independent interest.

\begin{example}
Consider the graph $F$ in Figure \ref{fig:1}. The reader can verify by
calculating the number of subgraphs of a given type that
$I(a_{12})(F)=9$, $I(a_{12}a_{13})(F)=15$, $I(a_{12}a_{13}a_{14})(F)=4$,
$I(a_{12}a_{13}a_{23})(F)=3$, $I(a_{12}a_{34}a_{45})(F)=52$ and
$I(a_{12}a_{23}a_{34})$ $(F)=16$. The algebraic dependence given by the product formula will turn out to be 
\begin{eqnarray}
I(a_{12})I(a_{12}a_{13})&=&2I(a_{12}a_{13})+2I(a_{12}a_{23}a_{34})+3I(a_{12}a_{13}a_{23})\\\nonumber 
&&+3I(a_{12}a_{13}a_{23})+I(a_{12}a_{34}a_{45})
\end{eqnarray}
and it shows that there is an algebraic dependence between these invariants. Indeed $9\cdot 15=2\cdot 15+2\cdot16+3\cdot 4+3\cdot 3+52$.
\end{example}

\begin{figure}[!htk]
\begin{center}
\includegraphics[width=3cm,height=3cm]{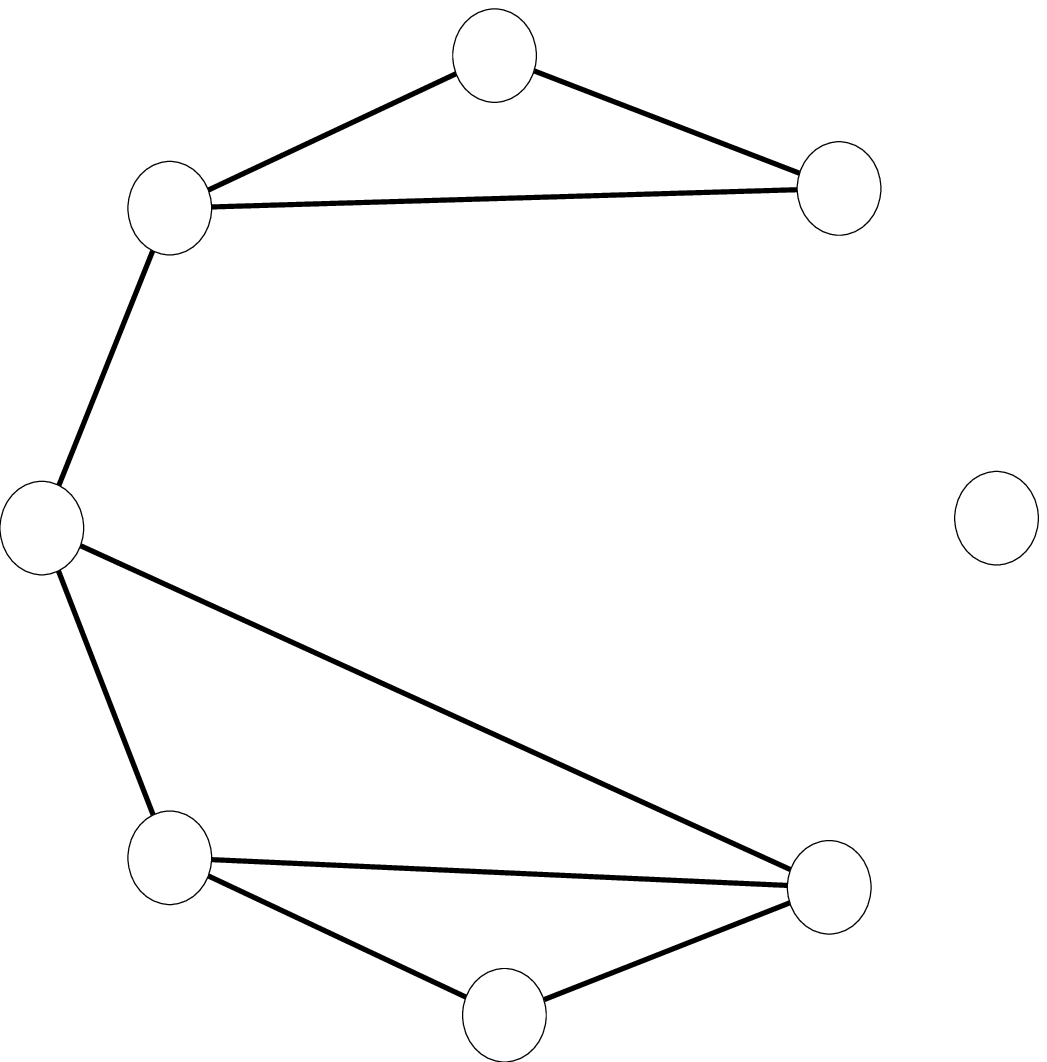}
\end{center}
\vspace{-0.5cm}
\caption{Graph $F$.}\label{fig:1}
\end{figure}

The graph isomorphism (GI) problem asks to determine, whether for any graphs
$A$ and $B$ there exists a permutation $\rho$ of vertices of the graph $A$
such that $\rho(A)=B$. There is no known polynomial-time algorithm for solving
GI and some results indicate that general GI might not belong to P
\cite{Arvind}, \cite{Koebler}. There are, however, several GI algorithms which perform very well on average \cite{McKay},\cite{Tinhofer}. If the vertex degree (i.e. the number of edges adjacent to a vertex) is bounded, then GI belongs to P \cite{Luks}. 

In section \ref{sec:se6} we calculate the upper and lower bounds for the minimal number of basic graph invariants required to prove graph isomorphism between two arbitrary graphs.

Many results presented here were proved independently by both of the authors and moreover had already been published previously by other mathematicians. The authors tried their best to provide a self-contained, rather complete, treatment of the subject, however the interested reader might have a look at the works of J.A. Bondy, W. L. Kocay,
V. B. Mnukhin, M. Pouz$\mathrm{\acute{e}}$t, B.D. Thatte and N. Thi$\mathrm{\acute{e}}$ry (\cite{Bondy}, \cite{Kocay}, \cite{Mnukhin}, \cite{Pouzet}, \cite{Thatte},\cite{Thiery1}, \cite{Thiery2}), each of them having its own distinct point of view (e.g. N. Thi$\mathrm{\acute{e}}$ry took a classical invariant theoretic approach in \cite{Thiery2}). We must note that
the classical invariants form a graded algebra unlike the invariants in this paper. This
is due to reduction $a_{ij}^2=a_{ij}$ since $a_{ij}\in\{0,1\}$ for simple
graphs. 

This paper is based on the product formula and what we
call the $G$-poset theory which is roughly the finite set system theory with a
permutation group. The product formula and $G$-poset theory are quite
essential in the reconstruction problem. In section $7$ we show a sufficient
condition for this conjecture to be true.

In the following section we calculate the product of basic graph
invariants $I(A)$ and $I(B)$. In section \ref{sec:se3} we show a couple of
examples and consequences of the product formula. In section \ref{sec:se4} we
show how all graph invariants can be written as a linear combination of the
basic graph invariants. In section \ref{sec:se5} we derive a simpler formula
for the product of two graph invariants. In section $7$ we study the minimal
set of generator/separator invariants. 

\section{Product Formula for Graph Invariants}\label{sec:se2}
Fleischmann's product formula for two orbit sums is not directly applicable to
graph invariants of simple graphs where $a_{ij}\in \{0,1\}$. We use a simple
example to show this. Let $G \subseteq S_n$ be any permutation group. For any number of vertices $n$, the permutation $(23)\in \mathrm{Stab}(a_{12}a_{13})$ but $(23) \notin \mathrm{Stab}(a_{12}^2a_{13})$, in fact $\mathrm{Stab}(a_{12}a_{13})=\langle (23), \mathrm{Stab}(a_{12}^2a_{13}) \rangle$. However $a_{12}^2a_{13}=a_{12}a_{13}$ if $a_{ij}\in{0,1}$ i.e. there is a reduction $a_{ij}^2=a_{ij}$. Thus 
\begin{equation}
\sum_{\rho \in G/\mathrm{Stab}(a_{12}^2a_{13})} a_{\rho(1)\rho(2)}^2a_{\rho(1)\rho(3)}=2\sum_{\rho \in G/\mathrm{Stab}(a_{12}a_{13})} a_{\rho(1)\rho(2)}a_{\rho(1)\rho(3)}.
\end{equation} 

In general by the \emph{the orbit-stabilizer theorem} 
\begin{eqnarray}\label{eq:kerroin}
\sum_{\rho \in G/\mathrm{Stab}(a_{i_1j_1}^{e_1}a_{i_2j_2}^{e_2}\cdots a_{i_dj_d}^{e_d})} a_{\rho(i_1j_1)}^{e_1}a_{\rho(i_2j_2)}^{e_2}\cdots a_{\rho(i_dj_d)}^{e_d}\\ \nonumber
=\frac{|\mathrm{Stab}(a_{i_1j_1}a_{i_2j_2}\cdots a_{i_dj_d})|}{|\mathrm{Stab}(a_{i_1j_1}^{e_1}a_{i_2j_2}^{e_2}\cdots a_{i_dj_d}^{e_d})|}\sum_{\rho \in G/\mathrm{Stab}(a_{i_1j_1}a_{i_2j_2}\cdots a_{i_dj_d})} a_{\rho(i_1j_1)}a_{\rho(i_2j_2)}\cdots a_{\rho(i_dj_d)},
\end{eqnarray} 
 making higher degree invariants redundant. Since $\prod_{i<j}^n a_{ij}$ is the highest degree monomial and it contains ${n \choose 2}$ variables, we observe that ${n \choose 2}$ provides an upper bound for the degree of basic graph invariants which are given by sums of type (\ref{eq:kerroin}).

Fleischmann's product formula for the product of orbit sums of monomials $A$ and $B$  over an arbitrary permutation group $G$ is
\begin{equation}\label{eq:eqp}
I(A)I(B)=\sum_{g\in[\mathrm{Stab}(A):G:\mathrm{Stab}(B)]} \frac{|\mathrm{Stab}(AgB)|}{|\mathrm{Stab}(A) \cap g\mathrm{Stab}(B)|}I(AgB),
\end{equation}
where $[G_1:G:G_2]$ denotes the cross-section of the group G with subgroups $G_1,G_2 \trianglelefteq G$ s.t. $G=\uplus_{g\in[G_1:G:G_2]}G_1gG_2$. This product formula applies directly to invariants of multigraphs where $a_{ij}$ are in a commutative algebra.

Let $\widehat{A}$ denote $a_{i_1j_1}^{e_1}\cdots a_{i_dj_d}^{e_d} \ \mathrm{mod} \langle a_{i_1j_1}^2-a_{i_1j_1},\ldots,a_{i_dj_d}^2-a_{i_dj_d} \rangle$, i.e. 
\begin{equation}
\widehat{a_{i_1j_1}^{e_1} \cdots a_{i_dj_d}^{e_d}}=a_{i_1j_1} \cdots a_{i_dj_d}.
\end{equation}
With this notation, we can express the Orbit Lemma as 
\begin{equation}
I(G)=\frac{|\mathrm{Stab}(\widehat{G})|}{|\mathrm{Stab}(G)|}I(\widehat{G}).
\end{equation}

 To get a product formula for graph invariants of simple graphs where $a_{ij}\in \{0,1\}$, we expand the terms in (\ref{eq:eqp}) by the formula (\ref{eq:kerroin}). This results in
\begin{eqnarray}
&&I(A)I(B)\\ 
&=&\sum_{g\in[\mathrm{Stab}(A):G:\mathrm{Stab}(B)]} \frac{|\mathrm{Stab}(AgB)|}{|\mathrm{Stab}(A) \cap g\mathrm{Stab}(B)|}I(AgB) \\ \nonumber
&=&\sum_{g\in[\mathrm{Stab}(A):G:\mathrm{Stab}(B)]} \frac{|\mathrm{Stab}(AgB)|}{|\mathrm{Stab}(A) \cap g\mathrm{Stab}(B)|}\frac{|\mathrm{Stab}(\widehat{AgB})|}{|\mathrm{Stab}(AgB)|}I(\widehat{AgB}) \\ \nonumber
&=&\sum_{g\in[\mathrm{Stab}(A):G:\mathrm{Stab}(B)]} \frac{|\mathrm{Stab}(\widehat{AgB})|}{|\mathrm{Stab}(A) \cap g\mathrm{Stab}(B)|}I(\widehat{AgB}).
\end{eqnarray}
This proves

\begin{theorem}\label{the:tulo}
The product formula for graph invariants $I(A)$ and $I(B)$, where $A,B$ are simple graphs, is
\begin{equation}\label{eq:tulo1}
I(A)I(B)=\sum_{g\in[\mathrm{Stab}(A):S_n:\mathrm{Stab}(B)]} \frac{|\mathrm{Stab}(\widehat{AgB})|}{|\mathrm{Stab}(A) \cap g\mathrm{Stab}(B)|}I(\widehat{AgB}),
\end{equation}
\end{theorem}
This formula is quite difficult to use but we can interpret the set of permutations $g\in[\mathrm{Stab}(A):S_n:\mathrm{Stab}(B)]$  by using colored graphs. We associate a monomial $A$ in the variables $a_{ij}$ with colored graphs by equating the color of the edge $(ij)$ with the exponent of the variable $a_{ij}$ in the monomial $A$. 
\begin{example}
The monomial $a_{13}a_{14}a_{24}^2a_{25}^2a_{35}^3$ corresponds to the graph in Figure \ref{fig:4}.
\end{example}

\begin{figure}[!htk]
\begin{center}
\includegraphics[width=3cm,height=4.3cm]{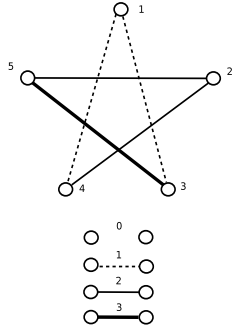}\label{fig:4}
\end{center}
\vspace{-0.5cm}
\caption{The colored graph corresponding to the monomial $a_{13}a_{14}a_{24}^2a_{25}^2a_{35}^3$. }
\end{figure}

Now consider the permutation group $G=S_n$ and all graphs $A\rho B$, where $\rho \in S_n$, such that the edges of $A$ have color 1, the edges of $B$ have color 2, the vertices of $B$ are permuted over all permutations and whenever two edges coincide the color of the edge is 3. Then the set of isomorphism classes of these colored graphs, denoted by $\mathcal{C}(A,B)$, corresponds to the set of monomials $\bigcup_{g\in[\mathrm{Stab}(A):S_n:\mathrm{Stab}(B)]} AgB$. The coloring of graphs corresponds to the modification of monomials $AgB$ such that all the variables in part $A$ are raised to the power $1$ and all the variables in $B$ are raised to the power $2$. 

\begin{proposition}
The map $\phi$ between the sets $\{AgB | g\in [\mathrm{Stab}(A):S_n:\mathrm{Stab}(B)]\}$ and $\mathcal{C}(A,B)$ coloring $AgB$ as above is bijective.
\end{proposition}
\begin{proof}

Since $\phi$ is clearly onto, all we have to show is that the set of colored graphs
\begin{equation}
\{AgB | g\in [\mathrm{Stab}(A):S_n:\mathrm{Stab}(B)]\}
\end{equation}
does not contain two elements $AgB$ and $AhB$ (where $g, h\in
[\mathrm{Stab}(A):S_n:\mathrm{Stab}(B)]$ and $h \neq g$) such that $\pi
AgB=AhB$ for some permutation $\pi \in S_n$. Suppose we had such a pair $\pi
AgB=AhB$. Because of the coloring we can recover the location of the edges of
$A$ in $AgB$, namely, every edge in $AgB$ with color $1$ or $3$
corresponds to an edge in $A$. This implies that the permutation $\pi \in
\mathrm{Stab}(A)$ since it maps $\pi A=A$. Also because of the coloring we
have $\pi gB=hB$ which implies that $\exists b\in \mathrm{Stab}(B)$ s.t. $\pi
g = hb$ as we can take $b=h^{-1}\pi g$.

 We can now solve $g=\pi^{-1}hb$ and so $g \in \mathrm{Stab}(A)h \mathrm{Stab}(B)$ which implies $g=h$ by the choice of $g,h \in  [\mathrm{Stab}(A):S_n:\mathrm{Stab}(B)]$. This is a contradiction and thus $\phi$ is an injective map.
 \end{proof}

Remark that the identity (\ref{eq:eqp}) gives
\begin{equation}
I(A) \sum_{\rho \in Sn/\mathrm{Stab}(B)}\rho B=\sum_{g\in[\mathrm{Stab}(A):S_n:\mathrm{Stab}(B)]} \frac{|\mathrm{Stab}(AgB)|}{|\mathrm{Stab}(A) \cap g\mathrm{Stab}(B)|}I(AgB).
\end{equation}

Kocay's lemma \cite{Kocay2} on the other hand says that the coefficient
$c_{A,B}^k$ equals the number of pairs $(C,D)$ where $C \cong A$, $D \cong B$ and $g_k = C \cup D$, for any representative of $G_k$. We can group these pairs according to the type of coloring they give and see that the groups are precisely the orbits of the coloring under the action of $Stab(G_k)$. Thus the two results are the essentially the same, any monomial in the resulting product being split into its distinct type of coloring in the Fleischmann formula.

\begin{example}
Consider the product $I(a_{12}a_{13})^2$, which can be calculated using Theorem \ref{the:tulo}.
\begin{eqnarray}
I(a_{12}a_{13})^2&=&4I(a_{12}a_{23}a_{34}a_{14})+2I(a_{12}a_{23}a_{34})+2I(a_{12}a_{23}a_{24}a_{34})\\
\nonumber
&&+2I(a_{12}a_{23}a_{24}a_{34})+6I(a_{12}a_{13}a_{14})+6I(a_{12}a_{13}a_{23})+I(a_{12}a_{13}).
\end{eqnarray}
The term $I(a_{12}a_{13}a_{14})$ arrives for instance from the monomial
$a_{12}a_{13}^2a_{14}$ $\equiv a_{12}a_{13}a_{14}$ $\mathrm{mod}$
$a_{13}^2-a_{13}$ . The coefficient of $I(a_{12}a_{13}a_{14})$ is $6$ because
the numerator $|\mathrm{Stab}$ $(a_{12}a_{13}a_{14})|=6$  and in the denominator the intersection of stabilizers $\mathrm{Stab}(a_{12}a_{13})=\{(1234),(1324)\}$ and $\mathrm{Stab}(a_{13}a_{14})=\{(1234),(1243)\}$ is the trivial group and thus the denominator is $1$.

There is a problem in the term $I(a_{12}a_{23}a_{24}a_{34})$. Instead of having this invariant with coefficient $4$ we have it split into two parts. This is because there are two non-isomorphic colorings for this graph in the product $I(a_{12}a_{13})^2$. See Figure \ref{fig:3} for these colorings. In section \ref{sec:se5} we will solve this problem. 
\end{example}

\begin{figure}[!htk]
\begin{center}
\includegraphics[width=3cm,height=1.3cm]{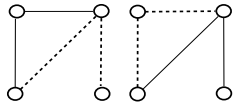}
\end{center}
\vspace{-0.5cm}
\caption{Two non-isomorphic colorings.}\label{fig:3}
\end{figure}

\section{Examples}\label{sec:se3}
The product formula (\ref{eq:tulo1}) describes connections between the numbers of different subgraph isomorphism classes of graphs. We use two examples to show these connections in explicit form.\\
\begin{example}
Let $g_1=I(a_{12})$, $g_2=I(a_{12}a_{13})$, $g_3=I(a_{12}a_{13}a_{23})$. The multiplication table of these invariants calculated using Theorem \ref{the:tulo} is given in Table \ref{tbl.mul}.

\begin{table}[!htk]
\begin{small}
\caption{Multiplication table for graph invariants with $n=3$.}\label{tbl.mul}
\begin{center}
\begin{tabular}{|l|lll|l}\hline
 & $g_1$ & $g_2$ & $g_3$ \\
 \hline 
$g_1$ & $g_1+2g_2$ & $2g_2+3g_3$ & $3g_3$ \\
$g_2$ & & $g_2+6g_3$ & $3g_3$ \\
$g_3$ & & & $g_3$ \\
\hline
\end{tabular}
\end{center}
\end{small}
\end{table}
We can see from Table \ref{tbl.mul} that the minimal generator set is
$\{g_1\}$ and the other invariants are given by the relations
$g_2=(g_1^2-g_1)/2$, $g_3=1/6g_1^3-1/2g_1^2+g_1$. Thus the transcendence
degree of $(k[a_{12},a_{13},a_{23}]/\langle a_{ij}^2-a_{ij} \rangle)^{S_3}$ is
one. The values of $g_1$ are limited by $0 \leq g_1 \leq 3$.\\
\end{example}
\begin{example}
Let $g_1=I(a_{12})$, $g_2=I(a_{12}a_{34})$, $g_3=I(a_{12}a_{13})$, $g_4=I(a_{12}a_{23}a_{34})$, $g_5=I(a_{12}a_{13}a_{14})$, $g_6=I(a_{12}a_{13}a_{23})$, $g_7=I(a_{12}a_{23}a_{34}a_{14})$, $g_8=I(a_{12}a_{23}a_{24}a_{34})$, $g_9=I(a_{12}a_{23}a_{34}a_{14}a_{13})$, $g_{10}=I(a_{12}a_{13}a_{14}a_{23}a_{24}a_{34})$. The first column of the multiplication table is in Table \ref{tbl.mul4}. We can solve for $g_2,g_4,g_6,g_7,g_8,g_9$ and $g_{10}$ in terms of $g_1,g_3$ and $g_5$ from Table \ref{tbl.mul4}, the solution is below.

\begin{table}[!htk]
\begin{small}
\caption{The first column of the multiplication table for graph invariants with $n=4$.}\label{tbl.mul4}
\begin{center}
\begin{tabular}{|l|l|l}\hline
 & $g_1$ \\
 \hline 
$g_1$ & $g_1+2g_2+2g_3$ \\
$g_2$ & $2g_2+g_4$  \\
$g_3$ & $2g_3+2g_4+3g_6+3g_5$  \\
$g_4$ & $3g_4+4g_7+2g_8$ \\
$g_5$ & $3g_5+g_8$  \\
$g_6$ & $3g_6+g_8$ \\
$g_7$ & $4g_7+g_9$ \\
$g_8$ & $4g_8+4g_9$ \\
$g_9$ & $5g_9+6g_{10}$ \\
$g_{10}$ & $6g_{10}$ \\
\hline
\end{tabular}
\end{center}
\end{small}
\end{table}

\begin{table}[!htk]
\begin{small}
\begin{tabular}{cc}
$g_2 =$ & $1/2g_1^2-1/2g_1-g_3$ \\
$g_4 =$ & $1/2g_1^3-3/2g_1^2-g_11g_3+g_1+2g_3$, \\
$g_6 =$ & $g_1g_3-2g_3-2/3g_1-g_5-1/3g_1^3+g_1^2$,\\ 
$g_8 =$ & $ g_1g_5-3g_5$, \\
$g_7 =$ & $-3/4g_1^3-1/2g_1g_5+3/2g_5+1/8g_1^4-3/4g_1$ \\
       & $-1/4g_1^2g_3+11/8g_1^2+5/4g_1g_3-3/2g_3$, \\
$g_{10} =$ & $-47/12g_1g_5-5/2g_1-5g_3+5g_5+137/24g_1^2$ \\
        & $+77/12g_1g_3-75/16g_1^3-1/12g_1^3g_5+g_1^2g_5$ \\
        & $-1/24g_1^4g_3+7/12g_1^3g_3-71/24g_1^2g_3$ \\
        & $+85/48g_1^4-5/16g_1^5+1/48g_1^6$, \\
$g_9 =$ & $-5/4g_1^4-1/2g_1^2g_5+7/2g_1g_5+1/8g_1^5-25/4g_1^2$ \\
       & $-1/4g_1^3g_3+35/8g_1^3+9/4g_1^2g_3-13/2g_1g_3-6g_5+3g_1+6g_3$.\\
\end{tabular}
\end{small}
\end{table}
\end{example}
For the reader familiar with general invariant theory (see \cite{Sturmfels}) we remark that by calculating the Gr$\mathrm{\ddot{o}}$bner basis of the relations in the multiplication table, we get syzygies describing completely the possible values of the graph invariants in  graphs with $n=4$. For instance $g_1$ satisfies the syzygy $g_1^7-21g_1^6+175g_1^5-735g_1^4+1624g_1^3-1764g_1^2+720g_1=0$ which has roots $0,1,2,\ldots,6$, determining the possible values for $g_1$.

The algebraic dependencies in the example above hold only if the number of vertices is $4$. It is easy, however, to construct general products.
\begin{lemma}\label{lem:lkjh}
 General products, i.e.\ products independent of the number of vertices, can be calculated by selecting $n \geq cv(A)+cv(B)$, where $cv(A)$ denotes the number of vertices in connection with the edges of the graph.
\end{lemma}
\begin{proof}
Notice that $cv(G)$ is the number of vertex-indices in the monomials of the graph invariant $I(G)$. The maximum number of distinct indices in any monomial $AgB$ is thus $cv(A)+cv(B)$.
 
The coefficient of $I(\widehat{AgB})$ in the product of $I(A)I(B)$ is  $\frac{|\mathrm{Stab}(\widehat{AgB})|}{|\mathrm{Stab}(A) \cap g\mathrm{Stab}(B)|}$. This remains the same when $n$ exceeds $cv(A)+cv(B)$. This can be seen by noticing that $\mathrm{Stab}(\widehat{AgB})=S_{n-cv(\widehat{AgB})} \times \mathrm{Stab}_{S_{cv(\widehat{AgB})}}(\widehat{AgB})$, where $S_n$ denotes the symmetric group and $\mathrm{Stab}_{S_{cv(\widehat{AgB})}}(\widehat{AgB})$ is the stabilizer of $\widehat{AgB}$ with respect to permutations of the connected vertices in $\widehat{AgB}$. Thus
\begin{equation}
|\mathrm{Stab}(\widehat{AgB})|=(n-cv(\widehat{AgB}))!|\mathrm{Stab}_{S_{cv(\widehat{AgB})}}(\widehat{AgB})|.
\end{equation}
 Next notice that 
\begin{equation}
\mathrm{Stab}(A) \cap g\mathrm{Stab}(B)=S_{n-cv(\widehat{AgB})}\times \left( \mathrm{Stab}_{cv(A)}(A) \cap g \mathrm{Stab}_{cv(B)}(B) \right)
\end{equation}
since no permutation in $\mathrm{Stab}(A) \cap g\mathrm{Stab}(B)$ can map a vertex in connection with the edges in $A$ or $B$ outside the set vertices in connection with the edges. Thus in the coefficient $\frac{|\mathrm{Stab}(\widehat{AgB})|}{|\mathrm{Stab}(A) \cap g\mathrm{Stab}(B)|}$ the terms $(n-cv(\widehat{AgB}))!$ appear both in the denominator and the numerator and cancel each other out.
\end{proof}

\begin{example}
The algebraic dependence 
\begin{eqnarray}
I(a_{12})I(a_{12}a_{13})=2I(a_{12}a_{13})+2I(a_{12}a_{23}a_{34})+3I(a_{12}a_{13}a_{23})\\
\nonumber
+3I(a_{12}a_{13}a_{14})+I(a_{12}a_{34}a_{45})
\end{eqnarray}
 is general holding for all graphs, not just for graphs with 5 vertices since $cv(a_{12})=2$ and $cv(a_{12}a_{13})=3$. 
\end{example}

\section{$G$-Posets and Mnukhin-Transforms}\label{sec:se4}
In this section we study orbit sums of monomials in a general context. We
generalize first the notion of $I(g_i)(g_j)$. Let $G\subseteq S_n$ be a
permutation group acting on the variables $x_1,\ldots,x_n$. Let $m^+$ denote
the orbit sum of the monomial $m$ over $G$, i.e. $\sum_{\rho \in
  G/\mathrm{Stab}(m)} x_1^{m_{\rho(1)}}\cdots x_n^{m_{\rho(n)}}$. 

To define basic invariants for multigraphs and more general objects we
introduce the following differential operator which plays the central role in
Cayley's $\Omega$-process in classical invariant theory \cite{Sturmfels}.

The differential operator corresponding to $m^+$ is defined as
\begin{equation}
\mathcal{D}_{m^+} := \frac{1}{m_1!m_2! \cdots m_n!}\sum_{\rho \in G/\mathrm{Stab}(m)} \frac{
  \partial^{|m|}}{\partial x_1^{m_{\rho(1)}} \cdots \partial x_n^{m_{\rho(n)}}}.
\end{equation}

The only difference with the original Cayley's operator is the coefficient
$\frac{1}{m_1!m_2! \cdots m_n!}$ which turns this operator into a Hasse
derivative. 

The value of this \emph{combinatorial invariant}, denoted by $I(m)(w)$ at the monomial $w$ is
\begin{equation}\label{eq:yle}
I(m)(w):=\{D_m(w)\}_{\mathbf{x}=1}
\end{equation}
where $m=x_1^{m_1}x_2^{m_2}\cdots x_n^{m_n}$ and $w=x_1^{w_1}x_2^{w_2}\cdots
x_n^{w_n}$. The reason for using the Hasse derivative is that $I(m)(m)$ should
be one maintaining the interpretation of counting subgraphs and defining an
unimodular Mnukhin-transform which we define shortly.

\begin{example}
Take $m=x_1x_2^2$ and $G=S_2$. Then we calculate
\begin{equation}
D_{m^+}=\frac{1}{2} \frac{\partial^3}{\partial x_1^2 \partial x_2}+\frac{1}{2}
\frac{\partial^3}{\partial x_1 \partial x_2^2}
\end{equation}
and
\begin{eqnarray}
&&\frac{1}{2} \frac{\partial^3}{\partial x_1^2 \partial x_2}x_1x_2^2+\frac{1}{2}
\frac{\partial^3}{\partial x_1 \partial x_2^2}x_1x_2^2 \\ \nonumber
&=&1.
\end{eqnarray}
Thus $I(m)(m)=\{x_1\}_{\mathbf{x}=1} = 1$.
\end{example}

\begin{example}
Take $m=x_1x_2^2, w=x_1^2x_2^2$ and $G=S_2$. Then 
\begin{eqnarray}
&&\frac{1}{2} \frac{\partial^3}{\partial x_1^2 \partial x_2}x_1^2x_2^2+\frac{1}{2}
\frac{\partial^3}{\partial x_1 \partial x_2^2}x_1^2x_2^2 \\ \nonumber
&=& 2x_2+2x_1.
\end{eqnarray}
Thus $I(m)(w)=\{2x_2+2x_1\}_{\mathbf{x}=1} = 4$.
\end{example}

\begin{lemma}
The invariants $I(a)$ coincide with the orbit sums $a^+$ if $a_i\in\{0,1\}$.
\end{lemma}
\begin{proof}
It is sufficient to consider one monomial $a \in Orb_G(a')$ and the
corresponding differential operator $\mathcal{D}_{a}$. The monomial $a$ at $b$ equals
\begin{equation}
b_1^{a_1} \cdots b_{N}^{a_N} = \prod_{i:a_i =1} b_i.
\end{equation}
The differential operator at $b$ equals
\begin{equation}
\mathcal{D}_a x_1^{b_1} \cdots x_N^{b_N} = \left(\prod_{i:a_i=1} b_i
  x^{b_i-1}\right) \prod_{i:a_i=0}x^{b_i}
\end{equation}
which equals $\prod_{i:a_i=1}b_i$ at $\mathbf{x}=1$.
\end{proof}

Let 
\begin{equation}
\mathcal{D}_i^{a_i}=\frac{1}{a_i!}\frac{\partial^{a_i}}{\partial_i^{a_i}}
\end{equation}
and notice that $\mathcal{D}_i^{a_i} x_i^{b_i}={b_i \choose a_i} x^{b_i-a_i}$.

This gives us an important clue how to find the linear combination of
differential operators $\sum_{k=0}^{\infty} c_k^{a_i}\mathcal{D}_i^{k_i}x_i^{b_i}$ s.t. the value at $x=1$ is
$b_i^{a_i}$.

The linear equation for the coefficients is 
\begin{equation}
B c^{a_i}=[0^{a_i},1^{a_i},\ldots,{\infty}^{a_i}]^T,
\end{equation}
where $B$ is the matrix defined by the elements $b_{ij}={i \choose j}$. This
is called \emph{the binomial transform} and it is well known to have the inverse
$B^{-1}$ defined by the elements $\hat{b}_{ij}=(-1)^{i-j}{i \choose j}$.

Thus we can solve
\begin{equation}
c_k^{a_i}=\sum_{j=0}^{\infty} (-1)^{k-j} {k \choose j}j^{a_i}
\end{equation}
yielding the desired linear combination
\begin{equation}
\{\sum_{k=0}^{\infty}\sum_{j=0}^{\infty} (-1)^{k-j} {k \choose j}j^{a_i}
\mathcal{D}_i^{k} x_i^{b_i}\}_{x_i=1}=b_i^{a_i}.
\end{equation}
Notice $\mathcal{D}_i^{b_i+1}x^{b_i}=0$ and thus we can restrict the infinite
sums to
\begin{equation}
\{\sum_{k=0}^{b_i}\sum_{j=0}^{b_i} (-1)^{k-j} {k \choose j}j^{a_i}
\mathcal{D}_i^{k} x_i^{b_i}\}_{x_i=1}=b_i^{a_i}.
\end{equation}

Finally we combine the results and obtain the following proposition.

\begin{proposition}
The monomial orbit sum $a^+$ equals the following linear combination of
combinatorial invariants:
\begin{equation}
a^+=\sum_{k_1,\ldots,k_N=0}^{\infty} \left(\sum_{j_1,\ldots,j_N=0}^{\infty}
\prod_{h=1}^{N} (-1)^{k_h-j_h}{k_h \choose j_h}j_h^{a_h}\right) I(x_1^{k_1} \cdots x_N^{k_N}).
\end{equation}
\end{proposition}

In the rest of this paper we restrict ourselves to exponents $m_i,w_i\in
\{0,1\}$. 

Having defined and shown some properties of the combinatorial invariants, it is
time to consider the underlying mathematical structure.

\begin{definition}
$G$-poset is a pair $(\E,G)$, where $\E$ is the set of
combinatorial invariants $I(g)$ or equally the set of equivalence classes of
non-negative vectors in $\mathbb{Z}_+^N$ with respect to $G\subseteq
S_N$ which is the permutation group acting on $\mathbb{Z}_+^N$.
\end{definition}

In this paper the non-negative vectors $v$ are always associated to their
monomial representations $x_1^{v_1} \cdots x_N^{v_N}$.

This notion is intended to stress the sociological behavior of the monomials
which means that each monomial corresponds to a basic invariant which can
be evaluated in all other monomials.

We say that a set of orbit sums of monomials $\E_G$ is a \textit{complete $G$-poset} with respect to the permutation group $G$ if the following holds. 
\begin{itemize}
\item[] For all monomials $w$ appearing in the orbit sums of the $G$-poset, all the submonomials $m\subseteq w$ appear also in some orbit sum in the $G$-poset.
\end{itemize}

We define the \emph{Mnukhin-transform} or the $M$-transform of $\E$ as a matrix
$E$ with entries $e_{ij}=I(m_j)(m_i)$, where $m_i$, $i=1\ldots N$ are all the
monomials representing the orbit sums in the $G$-poset $\E$. In \cite{Thatte}
B.D.Thatte calls this $\mathcal{N}$-matrix but the difference is that it
calculates the induced subgraphs of some graph $G$. On the other hand
V.B. Mnukhin calls this \emph{the orbit inclusion matrix} but we want to emphasize the interpretation as a transform \cite{Mnukhin}.

As with graph invariants, the value of $I(\pi(m_j))(\rho(m_i))$ is independent
with respect to the permutations $\pi,\rho\in G$ and thus we may choose an arbitrary monomials in the orbit sums containing $m_i$ and $m_j$ to calculate the value of $I(m_j)(m_i)$. We always label the monomial orbit sums in the $G$-poset so that $I(m_j)(m_i)=0$ if $i<j$. However, this does not uniquely specify the order of monomials of the same degree.

\begin{example}\label{ex:c1}
Let $G=1_G$ be the trivial group. Then the set of orbit sums $\E=\{x_1,x_2,x_3,x_1x_2,x_1x_3,x_2x_3,x_1x_2x_3\}$ is a $G$-poset. The $M$-transform is 
\begin{displaymath}
E=\left( \begin{array}{ccccccc}
1 & 0 & 0 & 0 & 0 & 0 & 0  \\
0 & 1 & 0 & 0 & 0 & 0 & 0  \\
0 & 0 & 1 & 0 & 0 & 0 & 0  \\
1 & 1 & 0 & 1 & 0 & 0 & 0  \\
1 & 0 & 1 & 0 & 1 & 0 & 0  \\
0 & 1 & 1 & 0 & 0 & 1 & 0  \\
1 & 1 & 1 & 1 & 1 & 1 & 1  \\ 
\end{array} \right).
\end{displaymath}
\end{example}

\begin{example}\label{ex:c2}
Let $G=\{(1),(12)\}\cong S_2$. The set of orbit sums $\E=\{x_1+x_2,x_3,x_1x_2,x_1x_3+x_2x_3,x_1x_2x_3\}$ is a $G$-poset. The $M$-transform is 
\begin{displaymath}
E=\left( \begin{array}{ccccc}
1 & 0 & 0 & 0 & 0  \\
0 & 1 & 0 & 0 & 0  \\
2 & 0 & 1 & 0 & 0  \\
1 & 1 & 0 & 1 & 0  \\
2 & 1 & 1 & 2 & 1  \\
\end{array} \right).
\end{displaymath}
\end{example}

In this paper our focus is on $G$-posets of graphs. They appear as a special case
when $G=S_n^{(2)}$, where $S_n^{(2)}$ refers to the representation of $S_n$
with the variables $a_{ij}$. The members correspond to isomorphism classes of
graphs. For instance the set of unlabeled graphs with $n$ vertices is a
complete $G$-poset denoted by $\E(n)$. Also the set of unlabeled forests and
the set of planar graphs are complete $G$-posets. We may say that a graph $G$-poset is composed of graphs even though we actually consider it as a $G$-poset of orbit sums.

There are (at least) two natural ways to restrict general graph $G$-posets: by limiting the number of vertices in connection with the edges and by limiting the number of edges in the graph. We use notation $\E(n,d)$ to denote the set of graphs with $cv(g)\leq n$ and $|g|\leq d$. As above we may omit the degree parameter by noticing $\E(n)=\E(n,\infty)=\E(n,{n \choose 2})$. Also $\E(\infty,d)=\E(2d,d)$.

Consider the invariant $I(g)(h)$. Here we may regard $g$ and $h$ either as the adjacency matrices of the graphs, monomials of the invariants $I(g)$ and $I(h)$ or the graph isomorphism class of type $g$ and $h$.

\begin{example}
Consider the $G$-poset $\E(4)$ with graph invariants $g_0=1$, $g_1=I(a_{12})$, $g_2=I(a_{12}a_{34})$, $g_3=I(a_{12}a_{13})$, $g_4=I(a_{12}a_{23}a_{34})$, $g_5=I(a_{12}a_{13}a_{14})$, $g_6=I(a_{12}a_{13}a_{23})$, $g_7=I(a_{12}a_{23}a_{34}a_{14})$, $g_8=I(a_{12}a_{23}a_{24}a_{34})$, $g_9=I(a_{12}a_{23}a_{34}a_{14}a_{13})$, $g_{10}=I(a_{12}a_{13}a_{14}a_{23}a_{24}a_{34})$. The $M$-transform is
\begin{displaymath}
E=\left( \begin{array}{ccccccccccc}
1 & 0 & 0 & 0 & 0 & 0 & 0 & 0 & 0 & 0 & 0 \\
1 & 1 & 0 & 0 & 0 & 0 & 0 & 0 & 0 & 0 & 0 \\
1 & 2 & 1 & 0 & 0 & 0 & 0 & 0 & 0 & 0 & 0 \\
1 & 2 & 0 & 1 & 0 & 0 & 0 & 0 & 0 & 0 & 0 \\
1 & 3 & 0 & 3 & 1 & 0 & 0 & 0 & 0 & 0 & 0 \\
1 & 3 & 1 & 2 & 0 & 1 & 0 & 0 & 0 & 0 & 0 \\
1 & 3 & 0 & 3 & 0 & 0 & 1 & 0 & 0 & 0 & 0\\ 
1 & 4 & 1 & 5 & 1 & 2 & 1 & 1 & 0 & 0 & 0 \\
1 & 4 & 2 & 4 & 0 & 4 & 0 & 0 & 1 & 0 & 0 \\
1 & 5 & 2 & 8 & 2 & 6 & 2 & 4 & 1 & 1 & 0\\ 
1 & 6 & 3 & 12 & 4 & 12 & 4 & 12 & 3 & 6 & 1 \\
\end{array} \right).
\end{displaymath}
We write indices from $0$ to $10$. Thus for example $e_{9,3}=I(g_3)(g_9)=8$. 
\end{example}

In complete \emph{multilinear $G$-posets} having the reduction $x_i^2=x_i$ we have the following simple and beautiful
theorem by V.B. Mnukhin \cite{Mnukhin}.

\begin{theorem}[Mnukhin]
Let $\E$ be a complete multilinear $G$-poset. The elements of $E^k, k \in \mathbb{Z}$ are given by
\begin{equation}\label{inverse}
E^{k}_{ij}=k^{|g_i|-|g_j|}e_{ij},
\end{equation}
where $|g|$ denotes the number of edges in the graph $g$ or the degree of the monomial $g$ and the $E^{k}_{ij}$ is the $ij^{\mathrm{th}}$ entry in the matrix $E^{k}$.
\end{theorem}

In particular the inverse $E^{-1}$ is given by $(-1)^{|g_i|-|g_j|}e_{ij}$. 

\section{Structure of M-transform}
We show that $M$-transform has some structure which allows at least some
redundancy in computation. Also we show one important fact about the rank of
certain minors of $M$-transforms.

\begin{lemma}\label{lem:anti}
Let $g$ be the structure of the monomial $a_{\tau_1}a_{\tau_2}\cdots a_{\tau_d}$. Then 
\begin{eqnarray}
I(\overline{g}):=\sum_{\rho \in S_n/Stab(a_{\tau_1}a_{\tau_2}\cdots a_{\tau_d}) }(1-a_{\rho(\tau_1)})(1-a_{\rho(\tau_2)})\cdots (1-a_{\rho(\tau_d)}) \\ \nonumber
=\sum_{a \subseteq g}(-1)^{|a|}\frac{I(a)(g)|Stab(a)|}{|Stab(g)|}I(a), \label{eq:kaks}
\end{eqnarray}
where $|a|$ is the number of edges in $a$ and the sum is over all unlabeled subgraphs of the graph $g$. 
\end{lemma}
\begin{proof}
 The number of terms in the first sum is $n!/|Stab(g)|$. Each of these terms contains $I(a)(g)$ monomials of the invariant $I(a)$. Since the number of monomials in $I(a)$ is $n!/|Stab(a)|$ we get the coefficient $\frac{I(a)(g)|Stab(a)|}{|Stab(g)|}$ for $I(a)$. 
\end{proof}

Notice that $I(g)(K_n \setminus g)=I(\overline{g})(G)$, where $g,G \in \E(n)$. Thus if
we know the values of $I(g)$ and its subinvariants $I(a), a\subseteq g$ in the graph $G$, we
know the value of $I(g)$ in $K_n \setminus G$. We can state this in a useful
manner by sorting the elements in $\E(n)$ in the following order.

Assume first that ${n \choose 2}+1$ is even. This is the number of different
degrees of graphs in $\E(n)$. For each graph $g$ of degree $d$ there is the corresponding
complement $K_n \setminus g$ of degree ${ n \choose 2}-d$. Thus by naming the
graphs as $g_1,\ldots,g_p$ up to the degree $({n \choose 2}-1)/2$ and the
remaining graphs $g_{|\E(n)|-i} \cong K_n \setminus g_i$ we get a nice labeling. 

Once we know the
$M$-transform up to the degree $({n \choose 2}-1)/2$, we can solve $e_{ij}$
for $i \geq |\E(n)|/2$ and for $j=0,1,\ldots |\E(n)|$ recursively by using
\begin{equation}
e_{ij}=\sum_{k \leq j} (-1)^{|g_k|} e_{jk}
\frac{|Stab(g_k)|}{|Stab(g_j)|} e_{|\E(n)|-i,k}.
\end{equation}

Let $N=|\E(n)|/2$. First solve  $e_{N+i,N}, i=1\ldots N$. Then $e_{N+i,N+1}$,
$e_{N+i,N+2}$ and so on until $e_{i,2N}$.

If ${n \choose 2}+1$ is odd, then there are invariants of degree ${n \choose
  2}/2$, whose complements are of same degree. Some graphs are even
  self-complement $g \cong K_n \setminus g$.

\begin{example}
Take $\E(4)$. Once we know $M$-transform up to the degree $3$ without $g_6=a_{12}a_{13}a_{23}$,
which is complement to $g_4=a_{12}a_{13}a_{14}$:
\begin{displaymath}
E^3=\left( \begin{array}{ccccccccccc}
1 & 0 & 0 & 0 & 0 & 0 & 0 & 0 & 0 & 0 & 0 \\
1 & 1 & 0 & 0 & 0 & 0 & 0 & 0 & 0 & 0 & 0 \\
1 & 2 & 1 & 0 & 0 & 0 & 0 & 0 & 0 & 0 & 0 \\
1 & 2 & 0 & 1 & 0 & 0 & 0 & 0 & 0 & 0 & 0 \\
1 & 3 & 0 & 3 & 1 & 0 & 0 & 0 & 0 & 0 & 0 \\
1 & 3 & 1 & 2 & 0 & 1 & 0 & 0 & 0 & 0 & 0 \\
\end{array} \right)
\end{displaymath}
we start by solving 
\begin{eqnarray}
e_{61}&=&\sum_{k=0}^1 (-1)^{g_k} e_{1k} \frac{|Stab(g_k)|}{|Stab(g_1)|} e_{K_n \setminus g_6,k} \\ \nonumber
&=& 4!/4 e_{5,0} - 4/4 e_{4,1} \\ \nonumber
&=& 3.
\end{eqnarray}
Next we solve 
\begin{eqnarray}
e_{71}&=&\sum_{k=0}^1 (-1)^{g_k} e_{1k} \frac{|Stab(g_k)|}{|Stab(g_1)|} e_{K_n \setminus g_7,k} \\ \nonumber
&=& 4!/4 e_{2,0} - 4/4 e_{2,1} \\ \nonumber
&=& 4
\end{eqnarray}
and so on up to $e_{10,1}$. Then we solve
\begin{eqnarray}
e_{62}&=&\sum_{k=0}^2 (-1)^{g_k} e_{2k} \frac{|Stab(g_k)|}{|Stab(g_2)|} e_{K_n \setminus g_6,k} \\ \nonumber
&=& 4!/4 e_{4,0} - e_{21}*4/4 e_{4,1} +e_{22}*e_{42}   \\ \nonumber
&=& 6-6+0 \\ \nonumber
&=& 0
\end{eqnarray}
and so on up to $e_{10,2}$. Then we continue with $e_{\cdot,3},e_{\cdot,4}
\ldots, e_{\cdot,5}$ similarly. Consider next
\begin{eqnarray}
e_{76}&=&\sum_{k=0}^6 (-1)^{g_k} e_{6k} \frac{|Stab(g_k)|}{|Stab(g_6)|} e_{K_n
  \setminus g_7,k} \\ \nonumber
&=& e_{60}(4!/3!)e_{20} - e_{61}(4/3!)e_{21} + e_{62}(4/3!)e_{22} \\ \nonumber
&=& 4-3*(4/6)*2+0 \\ \nonumber
&=& 0.
\end{eqnarray}
Here we have used $e_{60},e_{61}$ and $e_{62}$ which have been just calculated
above. Then continue up to $e_{10,10}$. 
\end{example}

Notice that while $e_{ij}$ is generally $\#P$-complete, the size of
the stabilizer $|Stab(g_i)|$ is polynomial time computable with a GI-oracle.

Consider next the case $G=1_G$. We sort the monomials in \emph{colex}-order
which allows to write easily the following recursive structure. The
colex-order means that we get all monomials of degree $\delta$ with
$n$ variables by concatenating the monomials of degree $\delta$ with $n-1$
variables with the $\delta-1$-degree monomials with $n-1$ variables multiplied
by the last variable $x_n$. For instance the list $x_1x_2,x_1x_3,x_2x_3$
extends to $x_1x_2,x_1x_3,x_2x_3,x_1x_4,x_2x_4,x_3x_4$. 

Denote by $E_\delta^\Delta(n)$ the minor in the
$M$-transform of the multilinear $G$-poset $\E$ with $n$ variables such that it contains the elements $e_{ij}$
s.t. $|g_j|=\delta$ and $|g_i|=\Delta$. Notice that with $\E(n)$ we are talking
about $E_\delta^\Delta({n \choose 2})$-minors.

\begin{lemma}\label{prangi1}
When $G=1_G$ and $1 \leq \delta \leq \Delta \leq n$, we have 
\begin{equation}
E_\delta^\Delta(n)= \left[ \begin{array}{ll} 
E_\delta^\Delta(n-1) & 0 \\ 
E_{\delta}^{\Delta-1}(n-1) & E_{\delta-1}^{\Delta-1}(n-1) \end{array} \right].
\end{equation}
Moreover the rank of $E_\delta^\Delta(n) \in \mathbb{Z}^{s \times t}$ is $\mathrm{min}(s,t)$.
\end{lemma}

\begin{proof}
Since $E_\delta^n(n)=[1,1,\ldots,1]^T \in \mathbb{Z}^{{n \choose \delta}}$ and
$E_\delta^\delta(n)=I$ the recursion is fully determined. 

The part $E_\delta^\Delta(n-1)$ corresponds to the monomials without the last
variable $x_n$. 

The part $E_{\delta}^{\Delta-1}(n-1)$ corresponds to the
monomials of degree $\Delta$ with the last variable $x_n$ and the evaluated
monomials of degree $\delta$ without the last variable $x_n$.

The part $E_{\delta-1}^{\Delta-1}(n-1)$ corresponds to the monomials of both
degrees with the last variable.

The rank is obviously as large as possible by the recursive structure.
\end{proof}

\begin{proposition}\label{prangi}
If $G$ is a general permutation group and $\E$ any complete $G$-poset with
respect to $G$, then $E_\delta^{\Delta}(n) \in \mathbb{Z}^{s \times t}$ has
rank $\mathrm{min}(s,t)$.
\end{proposition}
\begin{proof}
We start with $E_\delta^\Delta(n)$ and the trivial group. When we introduce the
symmetries from $G$, the original $E_\delta^\Delta(n)$ \emph{contracts} in the
following way:
\begin{itemize}
\item[i] All columns in the same orbit will be summed to one representative column.
\item[ii] All rows in the same orbit will be punctured, save one representative.
\end{itemize}

It is clear that once we begin with the matrix of maximal rank (with respect
to its dimensions), the contraction operation maintains the maximal rank.
\end{proof}

\section{Product Formula Based on M-transform}\label{sec:se5}
The inverse formula is useful in the calculation of products of graph invariants in the $G$-poset $\E$. Although the following product formula is general for all multilinear $G$-posets, we use the terminology of graph theory in this section. 
\begin{lemma}\label{lem:l1}
The polynomials in ${\mathbb C}[x_1,x_2,\ldots,x_n]/\langle x_i^2-x_i \rangle$ are in 1-1 correspondence with the values of the polynomials in $\{0,1\}^n$.
\end{lemma}
\begin{proof}
By induction we find the evaluation isomorphism of coefficients of the multilinear monomials and the values of the polynomials over $\{0,1\}^n$. Let $n=1$. Clearly the matrix 
\begin{displaymath}
E_1=
\left( \begin{array}{cc}
1 & 0 \\
1 & 1 
\end{array} \right)
\end{displaymath}
maps the coefficients of $p_1=c_0+c_1x_1$ in this order to the values of the polynomial over $\{0,1\}^n$. 
By adding the new variable $x_n$, the general polynomial $p_{n-1}$ becomes $p_n=p_{n-1}+p_{n-1}x_n$, where the part $p_{n-1}x_n$ has new coefficients. The corresponding evaluation isomorphism is obtained by $E_n=E_{n-1} \otimes E_1$ and is clearly invertible.
\end{proof}

Let $|\E|$ denote the number of members in the $G$-poset $\E$. Consider the vector $[I(g_i)(g)\cdot I(g_j)(g)]$, where $g$ runs through all the graphs in the $G$-poset and $g_i$ and $g_j$ are members of the $G$-poset. By calculating the inverse transform $c_{ij}=E^{-1}[I(g_i)(g) \cdot I(g_j)(g)],c_{ij}\in \mathbb{Z}^{|\E|}$ we obtain the linear combination $\sum_k c_{ij}^k I(g_k)$ of invariants in $\E$ such that $Ec_{ij}=[I(g_i)(g)\cdot I(g_j)(g)]$ and thus 
\begin{equation}
[I(g_i)(g)\cdot I(g_j)(g)]=\sum_k c_{ij}^k e_k= \sum_k c_{ij}^k[I(g_k)(g)], \ \forall g \in \E, \label{vekid}
\end{equation}
where $e_k$ is the $k^{\mathrm{th}}$ column of $E$. By Lemma \ref{lem:l1},
distinct polynomials in 
\begin{equation}
{\mathbb C}[x_1,x_2,\ldots,x_n]/\langle x_i^2-x_i \rangle
\end{equation}
obtain different values when evaluated in $\{0,1\}^n$. We know that the invariants obtain the same value in orbits of vectors in $\{0,1\}^n$ over the permutation group $G$. Since the orbits of $G$ divide the whole of $\{0,1\}^n$ into orbit sets covering the whole $\{0,1\}^n$, it is sufficient to check the members of $\E$ against the points $g \in \E$.

Thus by (\ref{vekid}) we have the following result.

\begin{theorem}[Mnukhin]\label{the:xyz}
 The product of two graph invariants in $\E$ equals the following linear combination of invariants
\begin{equation}\label{eq:29}
I(g_i)I(g_j)=\sum_{k=1}^N  c^k_{ij} I(g_k).
\end{equation}
where
\begin{equation}
c^k_{ij}=\sum_{h=1}^N (-1)^{|g_k|-|g_h|}e_{kh}e_{hi}e_{hj}
\end{equation}
\end{theorem}

As a special case Theorem \ref{the:xyz} gives a formula for the product of
graph invariants in a $G$-poset. By Lemma \ref{lem:lkjh} we know that by
selecting a sufficiently large $G$-poset, the product formula will hold.

We could have stated actually that the product of two members of the $G$-poset $\E$ equals the linear combination of the members in the $G$-poset. However, we hope that this generality is obvious for the reader and needs no further treatment.

 Let us restrict ourselves to the $G$-poset with $n$ vertices, $\E(n)$. Theorem \ref{the:tulo} also gives us a formula for the products of graph invariants i.e.
\begin{equation}\label{eq:31}
I(g_i)I(g_j)=\sum_{\rho\in[\mathrm{Stab}(g_i):S_n:\mathrm{Stab}(g_j)]} \frac{|\mathrm{Stab}(\widehat{g_i\rho g_j})|}{|\mathrm{Stab}(g_i) \cap \rho\mathrm{Stab}(g_j)|}I(\widehat{g_i \rho g_j}),
\end{equation}
where we have used $g_i,g_j$ instead of $A,B$ for clarity. Notice that this product is the product in $\E(n)$, since the monomials $g_i$ and $g_j$ consist of variables in adjacency matrices where the number of vertices is $n$.

Since the products (\ref{eq:29}) and (\ref{eq:31}) are equal, by collecting the coefficients isomorphic to $g_k$ we have
\begin{eqnarray}
\sum_{\rho\in[\mathrm{Stab}(g_i):S_n:\mathrm{Stab}(g_j)],\widehat{g_i\rho g_j}\cong g_k}  \frac{|\mathrm{Stab}(\widehat{g_i \rho g_j})|}{|\mathrm{Stab}(g_i) \cap \rho \mathrm{Stab}(g_j)|}\\ \nonumber
= \sum_{h=1}^N (-1)^{|g_k|-|g_h|}e_{kh}e_{hi}e_{hj}.
\end{eqnarray}

\begin{example}
Consider again the product $I(g_3)^2=I(a_{12}a_{13})^2$ in $\E(4)$. The new product formula gives the coefficient of $I(g_8)=I(a_{12}a_{23}a_{24}a_{34})$ by 
\begin{eqnarray}
c_{33}^8 &=&\sum_{h=1}^{10} (-1)^{e_{81}-e_{h1}} e_{8h}e_{h3}e_{h3} \\ \nonumber
&=&\sum_{h=3}^{10} (-1)^{e_{81}-e_{h1}} e_{8h}e_{h3}e_{h3} \\ \nonumber
&=&e_{83}e_{33}^2-e_{84}e_{43}^2-e_{85}e_{53}^2-e_{86}e_{63}+e_{87}e_{73}^2+e_{88}e_{83}^2 \\ \nonumber
&=&4-0-4\cdot 2^2-0+0+4^2\\ \nonumber
&=&4,
\end{eqnarray}
where we used the relation $|g_h|=e_{h1}$. The new product formula gives directly the coefficient $4$ for the $I(a_{12}a_{23}a_{24}a_{34})$ compared with the Example 5, where the coefficient was split up in two isomorphic terms.
\end{example}

In fact, consider any invariant $f(g)$ over a $G$-poset $\E$. Specifically, $f(g)$ is not necessarily a member of $\E$. For example $f(g)$ can be the maximal eigenvalue of the adjacency matrix of $g$, the chromatic number of $g$ or an integer representation of the canonical permutation of the graph \cite{McKay}.

 We can represent $f(g)$ as a linear combination of the basic graph invariants. This is done by evaluating $f$ over $\E$, which gives us a vector $v=[f(g_0),f(g_1),\ldots,f(g_N)]^T$, where $g_0,\ldots,g_N$ are the graphs in $\E$. The linear combination of the basic graph invariants equivalent to $f$ is now 
\begin{equation}
\sum_{i=0}^N c_i I(g_i),
\end{equation}
where $c=E^{-1}v$. Thus the study of graph invariants of a given $G$-poset of simple graphs can be reduced to the properties of the basic graph invariants.

\begin{proposition}
Let $c_{ij}^k$ be defined by 
\begin{equation}
c_{i,j}^k=\sum_{h=1}^N (-1)^{|g_k|-|g_h|}e_{kh}e_{hi}e_{hj}.
\end{equation}
Then the matrix $E$ can be recovered from $c_{ij}^k$ via
\begin{equation}
e_{ij}=c_{ij}^i.
\end{equation}
\end{proposition}
\begin{proof}
We have $e_{hi}=0$ if $h<i$ and $e_{ih}=0$ if $h > i$ and $e_{ii}=1$. Thus
\begin{eqnarray}
c_{ij}^i=\sum_{h=1}^N (-1)^{|g_i|-|g_h|}e_{ih}e_{hi}e_{hj} \\ \nonumber
=(-1)^{0}e_{ii}e_{ii}e_{ij}= \\ \nonumber
=e_{ij}.
\end{eqnarray}
\end{proof}

\begin{corollary}
The coefficients $c_{ij}^k$ are $\#P$-complete.
\end{corollary}
\begin{proof}
Evaluating $I(g)(h)$ i.e. counting the number of subgraphs isomorphic to $g$
in $h$ is $\#P$-complete \cite{Papadimitriou}.
\end{proof}

Next we calculate some identities which are needed later. The $G$-poset is
required to contain all multilinear monomials of degree $D+|g_i|$ which can be
formed by using the variables of the $G$-poset. For instance $\E(n)$ is allowed.
\begin{proposition}\label{pro:sigma}
 In the complete $G$-poset $(\E,G)$ containing all monomials of degree $D+|g_i|$ we have
\begin{equation}\label{eq:qwe}
I(g_i)\sum_{g\in \E,|g|=D} I(g)=\sum_{d=\max(D,|g_i|)}^{D+|g_i|} {|g_i| \choose |g_i|+D-d}\sum_{|g_k|=d} e_{ki}I(g_k).
\end{equation}
\end{proposition}
\begin{proof}
Write the product on the left-hand side of (\ref{eq:qwe}) as
\begin{equation}
\sum_{\rho \in S_n/\mathrm{Stab}(g_i)} \rho(A_i) \sum_{|A|=D}A,
\end{equation}
where the latter sum is over all monomials of $a_{ij}$ of degree $D$ and $A_i$ is the monomial of the invariant $I(g_i)$.

On the invariants of degree $D+|g_i|$ there is no overlapping of the variables in $A$ and $\rho(A_i)$. Thus the coefficient of any $I(g_k),|g_k|=D+|g_i|$  is $e_{ki}$. This is because for each monomial in $I(g_k)$ there are $e_{ki}$ choices for $\rho$.

When there are $\delta$ variables which have exponent $2$, there are ${|g_i|
  \choose \delta}$ different subsets of variables in $\rho(A_i)$ which result
in the same monomial in the reduction $a_{ij}^2=a_{ij}$. Thus
the coefficient is ${|g_i| \choose \delta}e_{ki}$. In equation (\ref{eq:qwe})
  we have used the notation $\delta=|g_i|+D-d$.
\end{proof}
There is a simple special case. Let $|h|=\sum_{|g_i|=1}I(g_i)$ denote the sum
of monomials of degree one where $h$ is understood as the graph where the
invariants are evaluated. 
\begin{corollary}\label{pro:eri}
In the complete $G$-poset $(\E,G)$ containing all monomials of degree $|g_i|+1$, we have
\begin{equation}
|h|I(g_i)=|g_i|I(g_i)+\sum_{|g_k|=|g_i|+1} e_{ki}I(g_k),
\end{equation}
where the sum is over basic invariants of degree $|g_i|+1$.
\end{corollary}
\begin{proof}
By summing over all monomials of degree one we have by the previous proposition
\begin{eqnarray}
I(g_i)\sum_{g\in \E,|g|=1} I(g)&=&\sum_{d=|g_i|}^{1+|g_i|} {|g_i| \choose |g_i|+1-d}\sum_{|g_k|=d} e_{ki}I(g_k) \\ \nonumber
&=&{|g_i| \choose 1}\sum_{|g_k|=|g_i|} e_{ki}I(g_k)+{|g_i| \choose 0}\sum_{|g_k|=|g_i|+1} e_{ki}I(g_k) \\ \nonumber
&=&|g_i|I(g_i)+\sum_{|g_k|=|g_i|+1} e_{ki}I(g_k).
\end{eqnarray}
\end{proof}


\section{Minimal Generator/Separator Invariants}\label{sec:se6}
In this section we focus on graph invariants solely. This restriction is
required by the structure of graphs which divide into connected and
unconnected graphs.

A \emph{generator set} $\mathcal{G}=\{g_1,g_2,\ldots,g_r\}$ for a set of graphs $\mathcal{H}$ is a set of graphs such that for each basic graph invariant $I(h), h\in \mathcal{H}$ there is a function $f$ such that $I(h)(x)=f(I(g_1)(x),I(g_2)(x),\ldots,I(g_N)(x))$ for all $x \in \mathcal{H}$. 

A \emph{separator set} $\mathcal{S}=\{g_1,g_2,\ldots,g_r\}$ for a set of graphs $\mathcal{H}$ is a set of graphs such that for each $x \in \mathcal{H}$ the vector $[I(g_1)(x),\ldots,I(g_r)(x)]$ has a distinct value. 

 As we saw in Example 6 in section \ref{sec:se3}, the invariant $I(a_{12})$
 forms the generator/separator set in $\E(3)$. Example 7 shows that in $\E(4)$
 the separator/generator set is $\{$$I(a_{12})$,$I(a_{12}a_{13})$,
 $I(a_{12}a_{23}a_{34})$$\}$. Note that these are all connected graphs. We can
 always choose the generators to be connected as we will see in Theorem \ref{the:con}.

The importance of studying the algebra of graph invariants and its generators towards reconstruction is due to the fact that in this algebra, there is no distinction between the notion of a separator set and of a generator set.
This result is originally due to Mnukhin \cite{Mnukhin} but we prove it
here for completeness.

\begin{theorem}[Mnukhin]\label{the:msepa}
For simple graphs any generator set is also a separator set.
\end{theorem}

\begin{proof}
First we show that a separator set $\{g_1,\ldots,g_r\}$ is also a generator set. 

By definition the vector $[I(g_1),\ldots,I(g_r)]$ gets a distinct value for
all graphs $h$ in the $G$-poset. Thus we can define the function $f$ to map the
vector $[I(g_1)(h)$, $\ldots,I(g_r)(h)]$ to $I(g)(h)$ for every $h\in \E$, where $I(g)$ is an arbitrary graph invariant and we are done.

To show that a generator set $\{g_1,\ldots,g_r\}$ is also a separator set it suffices to show that any separator set of invariants can be written as a function of the generators. Let $f_h$ be a function generating the invariant $I(h)$ and let $h_1,\ldots,h_s$ be any separator set. Now the vector $[f_{h_1},\ldots,f_{h_s}]$ separates all the graphs in the $G$-poset. 

\end{proof}

\begin{theorem}\label{the:con}
Connected graphs in the $G$-poset generate/separate the whole $G$-poset.
\end{theorem}
\begin{proof}
Let  $\{G_1,G_2,\ldots,G_r\}$ be the connected graphs in the $G$-poset. The
result follows from the following reconstruction algorithm which maps the
input graph $G$ into the representation $\{n_1G_1,n_2G_2,\ldots,n_rG_r\}$, where $n_i$ is the number of isolated components of type $G_i$ in the input graph. By isolated component we understand that the edges of one component are not connected to any other component. Let  $[G_1,G_2,\ldots,G_r]$ be ordered s.t. $deg(G_i) \geq deg(G_{i-1})$.

\vspace{0.5cm}
\begin{tabular}{ccc}
\hline
\textsc{Algorithm} & & \\
 & Input: &  $G$,  $[G_1,G_2,\ldots,G_r]$\\

 & 1 & $m:=r$ \\
 & 2 & Set $n_m=I(G_m)(G)-\sum_{k=m+1}^r n_kI(G_m)(G_k)$.\\
 & 3 & Set $m = m-1$. If $m>0$ goto 2.\\
 & 4 & Print $n_1,n_2,\ldots,n_r$.\\
\hline
\end{tabular}
\vspace{0.5cm}

The only step requiring some explanation is $2$. For each occurrence of the graph $G_i$ in the $G$, the invariant $I(G_i)(G)$ increases by one. However $I(G_i)(G)$ increases also in the connected graphs of higher degree  $G_{i+1},G_{i+2},\ldots,G_{r}$ and these must be subtracted.
\end{proof}

\begin{corollary}
The number of minimal generator/separator invariants is at most the number of
connected graphs in the $G$-poset.
\end{corollary}

Let $A$ and $B$ be the monomial representations of the graphs $A$ and $B$. By
the \emph{disjoint union of graphs} $A \dunion B$ we mean the isomorphism
class of graphs $C=A\rho B$ such that for a suitable 0 $\rho$ of the
vertices of $B$, the edges of the graph $A$ are not connected to the edges of
$\rho B$, if such a $\rho$ exists. From now on we use the notation
$n_1G_1\dunion \cdots \dunion n_rG_r$ to denote the graph formed by the disjoint union of $n_i$ $G_i$ graphs for each $G_i$, $i=1\ldots r$.

We saw above that in small $G$-posets like $\E(3)$ and $\E(4)$ even some
 connected graphs can be generated by a smaller number of connected
 graphs. However this result does not hold for arbitrarily large $G$-posets. It
 is possible to define infinitely large $G$-posets which only contain connected
 components of some finite set $\{G_1,G_2,\ldots,G_r\}$ but multiple times. 
 We use the notation $\langle G_1,\ldots,G_r \rangle$ to denote the complete
 $G$-poset which contains all graphs of the form $n_1G_1 \coprod \cdots \coprod
 n_rG_r$, $n_i \in \mathbb{Z}^+$.

The following theorem explains what happens in this case when the $G$-poset becomes large i.e. the number of edges grows without bound.

\begin{theorem}
Let $G_1,\ldots,G_r$ be the connected graphs of degree $\leq d$, where $d \geq
1$. Then there are graphs $T$ and $U$ of degree at most $(d+1)(2^d-1)$ which cannot be separated/generated
by $I(G_1),\ldots,I(G_r)$. 
\end{theorem}

Since all unconnected invariants of degree $\leq d$ can be determined when the
connected invariants are known, the $T$ and $U$ are consequently inseparable by all graph invariants of
degree $d$ and less.

\begin{proof}
First select a connected graph $G_{r+1}$ of degree $d+1$ not appearing in the set
$\{G_1,\ldots,G_r\}$. We may safely assume that the degrees of
$\{G_1,\ldots,G_r\}$ are greater or equal to one since the constant invariant
does not help in separation.

Let $\EErr=\langle n_1G_1\dunion \cdots \dunion n_{r}G_{r} \rangle$ and
$\EEr=\langle n_1G_1\dunion \cdots \dunion n_{r+1}G_{r+1} \rangle$ be
$G$-posets generated by the connected graphs $G_1,\ldots,G_{r-1}$ and
$G_1,\ldots,G_{r}$ correspondingly. Let $\{g_1,g_2,\ldots,g_R\}$ denote the members of $\EErr$ of degree $\leq
d$. Below we will show that $I(G_{r+1})$ is independent of $\{I(g_1),I(g_2),\ldots,I(g_{R})\}$ when the $G$-poset is sufficiently large. 

The idea of the proof is to generate large graphs $T$ and $U$ s.t. they cannot
be separated by the $r$ connected graph invariants $I(G_1),\ldots,I(G_{r})$. This
implies that there is no function $f$
s.t. $I(G_{r+1})=f(I(G_1),\ldots,I(G_{r}))$. If $I(G_1),\ldots,I(G_r)$ are all the
connected invariants of degree $d$ and less, $I(G_{r+1})$ can neither be written as a function of $I(g_1),I(g_2),\ldots$ since these are generated by $\{I(G_1),I(G_2),\ldots,I(G_{r})\}$.

Let 
\begin{equation}
c=[I(g_1)(G_{r+1}),I(g_2)(G_{r+1}),\dots,I(g_R)(G_{r+1})]E^{-1},
\end{equation}
where $E$ is the $M$-transform of the $G$-poset $\EErr$ up to degree $d$. Divide $c$ into positive and negative parts s.t. $c=c^{+}-c^{-}$ and $\forall i:$ $c^{+}_i\geq 0$, $c^{-}_i\geq 0$. The coefficients $c$ are selected so that
\begin{equation}
\forall i=1..r:\sum_{k=1}^{R}c_kI(G_i)(g_k)=I(G_i)(G_{r+1})
\end{equation}
For a connected graph $A$ we have
\begin{equation}
I(A)(B \dunion C)=I(A)(B)+I(A)(C),
\end{equation}
where $\dunion$ denotes the disjoint union of two graphs i.e. the edges of $B$ and $C$ are not connected in $B \dunion C$.

Since $G_1,\ldots,G_r$ are connected, we have
\begin{eqnarray}
\forall i=1..r:\sum_{k=1}^{R}c_k^+I(G_i)(g_k) \\ \nonumber
=I(G_i)(\coprod_{k=1}^R c_k^+ g_k) \\ \nonumber
=I(G_i)(G_{r+1} \dunion \coprod_{k=1}^R c_k^- g_k),
\end{eqnarray}
where the coefficients $c$ in the unions denote the multiplicity of the corresponding graph. Thus we have found graphs
\begin{eqnarray}
T=\coprod_{k=1}^R c_k^+ g_k, \\ \nonumber
U=G_{r+1} \dunion \coprod_{k=1}^R c_k^- g_k
\end{eqnarray}
such that they cannot be distinguished with invariants $I(G_1),I(G_2),\ldots,I(G_{r})$. 

It remains to calculate an upper bound for $d=max(deg(T),deg(U))$. Clearly
$deg(T) \leq \sum_{k=1}^R |c_k||g_k|$. We expand this

\begin{eqnarray}
deg(T) \leq \sum_{k=1}^R |\sum_{h=1}^R (-1)^{|g_h|-|g_k|}I(g_k)(g_h)I(g_h)(G_{r+1})| |g_k| \\ \nonumber
=|\sum_{h=1}^R (-1)^{|g_h|}\sum_{k=1}^RI(g_k)(g_h)I(g_h)(G_{r+1})| |g_k| \\ \nonumber
=|\sum_{h=1}^R (-1)^{|g_h|} \sum_{\Delta_1=1}^{|g_h|} {|g_h| \choose \Delta_1}\Delta_1 I(g_h)(G_{r+1})| \\ \nonumber
= |\sum_{\Delta_2=1}^{d} (-1)^{\Delta_2} \sum_{\Delta_1=1}^{\Delta_2} {\Delta_2 \choose \Delta_1} \Delta_1 {|G_{r+1}| \choose \Delta_2}|.
\end{eqnarray}

The last sum equals
\begin{eqnarray}
 {|G_{r+1}| \choose \Delta_2}\sum_{\Delta_1=1}^{\Delta_2} \Delta_1{\Delta_2 \choose \Delta_1}\\ \nonumber
= {|G_{r+1}| \choose \Delta_2} \sum_{\Delta_1=1}^{\Delta_2} \Delta_2 \frac{(\Delta_2-1)!}{(\Delta_2-\Delta_1)!(\Delta_1-1)!} \\ \nonumber
= {|G_{r+1}| \choose \Delta_2}\Delta_2  2^{\Delta_2-1}
\end{eqnarray}
and thus the whole sum is
\begin{eqnarray}
|\sum_{\Delta_2=1}^{d} (-1)^{\Delta_2} \Delta_2 {|G_{r+1}| \choose \Delta_2} 2^{\Delta_2-1}| \\ \nonumber
=|G_{r+1}| |\sum_{\Delta=1}^{d} (-1)^\Delta \frac{(|G_{r+1}|-1)!}{(|G_{r+1}|-\Delta)!(\Delta-1)!}2^{\Delta-1}| \\ \nonumber
=|G_{r+1}| |\sum_{\Delta=1}^{d} (-1)^\Delta {|G_{r+1}|-1 \choose \Delta -1} 2^{\Delta-1}| \\ \nonumber
=|G_{r+1}| |\sum_{\Delta=1}^{N} (-1)^\Delta {d \choose \Delta-1} 2^{\Delta-1}|.
\end{eqnarray}
Let $S(d):=\sum_{\Delta=1}^{N} (-1)^\Delta {d \choose \Delta-1}
2^{\Delta-1}$. Obviously $S(1)=-1$. Also
\begin{eqnarray}
S(d+1)&=&\sum_{\Delta=1}^{d+1} (-1)^\Delta {d+1 \choose \Delta -1 }2^{\Delta
  -1} \\ \nonumber
&=&\sum_{\Delta=1}^{d+1} (-1)^\Delta \left( {d \choose \Delta -1 } +{d \choose
  \Delta-2} \right) 2^{\Delta - 1} \\ \nonumber
&=&\left( \sum_{\Delta=1}^d (-1)^\Delta {d \choose \Delta-1} 2^{\Delta -1}
\right) + (-1)^{d+1}2^d \\ \nonumber 
&-&2 \left( \sum_{\Delta=1} (-1)^{\Delta+1} { d \choose \Delta -2} 2^{\Delta
    -2} \right) \\ \nonumber
&=& S(d) + (-1)^{d+1}2^d -2 S(d) \\ \nonumber
&=& -S(d)+(-1)^{d+1}2^d.
\end{eqnarray}
The solution to the recursion $S(d+1)=-S(d)+(-1)^{d+1}2^d$ with initial
constraint $S(1)=-1$ is $S(d)=(-1)^d(2^d-1)$. 

Thus we have $|T| \leq |G_{r+1}|(2^d-1)$. Since the invariant computing the
degree of graphs is generated by all invariants of degree $1$, the degree of
$U$ must be equal to the degree of $T$.
\end{proof}

We have formulated the proof so that it is easy to consider the case where
$G_{r+1}$ could be chosen to be of degree $\leq d$. In other words, if the
$G$-poset misses some connected invariant of degree $\leq d$, then the result
applies with the degree bound $|G_{r+1}|2^d$, where the additional term
$|G_{r+1}|$ is included since we are not sure anymore if the $G$-poset $\EErr$
can generate the invariant computing the degree of $T$. Thus the degree of $U$
is the upper bound. 

The $G$-poset $\EErr$, however, must be complete to prove this upper bound. If it is not complete, $T$
and $U$ are still inseparable but their degree is possibly harder to estimate
since we don't fully understand the corresponding $E^{-1}$. It is still unimodular,
however, and the degree is finite.

Consider now the infinite $G$-poset of all simple graphs $\E(\infty)$. By the
reasoning above we get the next corollary.
\begin{corollary}
In $\E(\infty)$ the minimal generator/separator set is the set of all connected
invariants.
\end{corollary}

The next corollary explains the result in terms of weight enumeration functions. Let $f(x)$ be the number of connected graphs of degree $x$ in the $G$-poset of interest. Define then $F(x)=\sum_{y=1}^x f(y)$. For example if we consider the $G$-poset of all graphs of degree $1,2,\ldots$ we have the following corollary.
\vspace{0.5cm}
\begin{corollary}
 Let $\E$ be a $G$-poset with all graphs of degree $(d+1)(2^d-1)$ and less. Then the size of the minimal generator/separator set is at least $F(d)+1$.
\end{corollary}

\begin{corollary}
The set of minimal separators/generators increases without limit as the number of vertices $n$ approaches infinity.
\end{corollary}

In the following we describe an upper bound for the generator/separator
invariants in $\E(n)$.

\begin{theorem}
The size of the minimal generator/separator set of $\E(n)$ is at most $h(\lfloor {n
  \choose 2}/2 \rfloor)$, where $h(d)$ is the number of graphs in $\E(n)$ with $d$ edges.
\end{theorem}

\begin{proof}
Use Corollary \ref{pro:eri}
\begin{equation}\label{eq:lhd}
(|h|-|g_i|)I(g_i)=\sum_{|g_k|=|g_i|+1} e_{ki}I(g_k)
\end{equation}
to solve the invariants of degree $d+1$ by all invariants of degree $d$. By Proposition
\ref{prangi} at least the same amount of invariants can be solved as there are
invariants of degree $d$. After degree $\lfloor {n \choose 2}/2 \rfloor$,
the system is fully or overdetermined.
\end{proof}

\begin{proposition}\label{Ulam}
The Ulam's reconstruction conjecture is true for graphs with $n+1$ vertices
and $d$ edges if
\begin{equation}
h_{n+1}(d)-h_{n+1}(d-1) \leq h_n(d) \ \forall \ d \leq \lfloor {n+1 \choose
  2}/2 \rfloor,
\end{equation}
where $h_n(d)$ is the number of unlabeled graphs with $d$ edges and $n$ vertices.
\end{proposition}
\begin{proof}
The invariants of degree $1$ in $\E(n+1)$ are obviously generated by the invariants of
degree $1$ in $\E(n)$ when $n \geq 2$. If the invariants of degree $d-1$ in
$\E(n+1)$ are generated by the invariants of degree $d-1$ and less in $\E(n)$,
we have the above system of equations to solve the remaining
$h_{n+1}(d)-h_{n+1}(d-1)$ invariants of degree $d$ once we know the invariants
of degree $d$ in $\E(n)$. The only question is whether the invariants in
$\E(n+1) \setminus \E(n)$ of degree $d$ are linearly independent in the system
of equations.

In analogous fashion to Lemma \ref{prangi1} and Proposition \ref{prangi} we consider
the minors $E_{d-1,\leq v}^{d,v}$ of $M$-transforms, where the graphs of
degree $d-1$ with $\leq v$ connected vertices are evaluated in graphs of degree $d$ with $v$ connected vertices.

If $E_{d-1, \leq v}^{d,v}$ has full rank and the hypothesis 
\begin{equation}
h_{n+1}(d)-h_{n+1}(d-1) \leq h_n(d) \ \forall \ d \leq \lfloor {n+1 \choose
  2}/2
\end{equation}
holds, then the system (\ref{eq:lhd}) is fully/overdetermined for the graphs $g$ with
parameters $|g|=d$, $cv(g)=n+1$ in terms of graphs $h$ with parameters
$|h|=d-1$, $cv(h) \leq n$.

We start with the trivial group and obtain the following recursive structure after realizing that graphs $h$ with
parameters $|h|=d-1$, $cv(h) \leq n$ can be obtained simply by puncturing the
variables, one at a time, in graphs $g$ with $|g|=d$, $cv(g)=n+1$. It does not
really matter in which order the variables $a_{ij}$ are ordered.
\begin{equation}
E_{d-1,\leq v}^{d,v}({v \choose 2})= \left[ \begin{array}{ll} 
E_{d-1,\leq v}^{d,v}({v \choose 2}-1)  & 0 \\ 
E_{d-1,\leq v}^{d-1,v}({v \choose 2}-1) &
E_{d-2,\leq v}^{d-1,v}({n \choose 2}-1) \end{array} \right].
\end{equation}
This recursive structure together with the similar initial forms as in Lemma
\ref{prangi1} imply that the system has full rank. Then apply contractions
given in Lemma \ref{prangi} and conclude that $E_{d-1,\leq v}^{d,v}$
with the permutation group $G$ has also full rank.
\end{proof}

We list here some computational data on enumerators for the reader to see
how the condition of Proposition \ref{Ulam} holds on small graphs. The entries
in the table are $-h_{n}(d)+h_{n-1}(d)+h_{n}(d-1)$ which should be non-negative.
\begin{table}
\begin{center}
\begin{tabular}{r|rrrrrrrrr}
$d \setminus n$ &4&5&6&7&8&9&10&11&12\\
\hline
2&0&1&1&1&1&1&1&1&1\\
3&0&1&1&2&2&2&2&2&2\\
4&&0&2&4&4&5&5&5&5\\
5&&1&0&4&8&10&10&11&11\\
6&&&0&1&9&18&23&25&25\\
7&&&3&0&6&30&49&60&65\\
8&&&&-8&-9&24&82&133&157\\
9&&&&-13&-50&-24&96&265&385\\
10&&&&-2&-113&-203&-29&410&878\\
11&&&&&-169&-635&-738&173&1678\\
12&&&&&-201&-1431&-3018&-2237&1779\\
\hline
\end{tabular}
\caption{The difference $-h_{n}(d)+h_{n-1}(d)+h_{n}(d-1)$ for simple graphs.}\label{tauluv}
\end{center}
\end{table}

As we can see in Table \ref{tauluv}, the system is sufficient for graphs with
small number of edges. As expected, the system of equations (\ref{eq:lhd}),
where we multiply by only $I(a_{12})$ is insufficient for graphs with many edges.

\section{Open Problems}
The $M$-transform plays central role in the results of this paper. Just like
all graph invariants are linear combinations of the basic graph invariants,
all knot invariants are linear combinations of the Vassiliev's knot invariants.

\newtheorem{op1}{Problem}[section]
\begin{op1}
Can you apply the theory of $G$-posets Vassiliev's knot invariants and find
lower and upper bounds for knot invariants?
\end{op1}

\newtheorem{op2}[op1]{Problem}
\begin{op2}
Can you prove Ulam's reconstruction conjecture as in Proposition \ref{Ulam} by
using more invariants in the products?
\end{op2}


\end{document}